\documentclass[12pt]{article}
\usepackage{amssymb}

\usepackage{amsmath}
\usepackage{tikz}
\usepackage{amscd}
\usepackage{latexsym}
\usepackage{mathrsfs}
\usepackage{enumitem}
\usepackage{lipsum}
\usepackage[noadjust]{cite}
\usepackage{lineno} 
\usepackage{mathtools}
\usepackage[hidelinks]{hyperref}

\usepackage{indentfirst}
\newtheorem{lemma}{Lemma}[section]
\newtheorem{theorem}[lemma]{Theorem}

\newcommand{\qed}{\hfill $\Box$}
\newenvironment{proof}{ \emph{proof.}}{\qed}

\def\bc{\begin{center}}
\def\ec{\end{center}}

\numberwithin{equation}{section}

\begin{document}
\title{On the generation of the symmetric group by index relations of three prefix reversals}
\author{Yexiao Lei, Jinxing Zhao\thanks{Corresponding author, E-mail address:zhjxing@imu.edu.cn.},  Ziyu Zhu, Yang Wu
\\{ \small \it School of Mathematics Sciences, Inner Mongolia University, Hohhot, P. R. China.}}
\date{}
\maketitle
\vskip -8pt
\baselineskip 16pt
\newcommand\blfootnote[1]{%
\begingroup
\renewcommand\thefootnote{}\footnote{#1}%
\addtocounter{footnote}{-1}%
\endgroup}
\begin{center}
\begin{minipage}{140mm}

\noindent{\bf Abstract:} The cubic pancake graphs are Cayley graphs over the symmetric group \(S_{n}\) generated by three prefix reversals. There is the folllowing open problem: characterize all the sets of three prefix reversals that generate \(S_{n}\). Blanco and Golubyatnikov give a complete solution of the problem when any of the two smallest or the two largest lengths but n are included in a triple of prefix reversals. According to their conjectures, we give some conditions implying a triple of prefix reversals does not generate \(S_{n}\). Moreover, Using the symmetric groups that can already be generated by three prefix reversals, we can derive some recurrence conditions.

\end{minipage}
\end{center}

\section{Introduction}
A classical and celebrated result \cite{Dixon1968} states that two random permutations of a set of \(n\) elements almost surly generate either the symmetric or the alternating group of degree \(n\), and similar results hold when more generators are considered. In particular, the symmetric group can always be generated by the transposition \(\left(1\ 2\right)\) together with the \(n\)-cycle \(\left(1\ 2\ \cdots \ n\right)\). Two well-known generating sets of \(n-1\) transpositions are the Coxeter generators \(\left\{(i\ (i+1)\ |\ 1\le i<n\right\}\) and the star generators \(\left\{(1\ i)\ |\ 1< i\le n\right\}\) .

Prefix reversals are a classical family of involutions arising in permutation sorting and the study of Cayley graphs. Let \(S_n\) denote the symmetric group on \(\{1,\ldots,n\}\). For each \(2\le i\le n\), define the prefix reversal \(r_i\in S_n\) by
\[
r_i(j)=
\begin{cases}
i+1-j, & 1\le j\le i,\\
j, & i<j\le n.
\end{cases}
\]
Equivalently,
\[
r_i=\prod_{j=1}^{\lfloor i/2\rfloor}(j\;\;i+1-j).
\]
Thus, \(r_i\) reverses the first \(i\) positions and fixes all remaining positions. The problem of sorting a permutation using as few prefix reversals as possible has attracted considerable attention. Gates and Papadimitriou \cite{William1979} established bounds on the maximum number of reversals required, and Chitturi et al. \cite{Chitturi2009} subsequently improved the upper bound. Bulteau, Fertin, and Rusu \cite{Laurent2015} proved that finding an optimal sorting sequence is NP-hard. These results concern the length of a sorting sequence when all prefix reversal lengths are available.

Restricting the allowed reversal lengths leads to a related generation problem: which subsets of prefix reversals suffice to generate every permutation? This question has a natural interpretation in terms of Cayley graphs. The pancake graph is the Cayley graph
\[
P_n=\operatorname{Cay}(S_n,\{r_2,\ldots,r_n\}),
\]
and retaining only the edges associated with a subset of reversals gives a spanning subgraph that is connected precisely when those reversals generate \(S_n\). Bass and Sudborough \cite{Douglas2003} studied restricted prefix reversals and the corresponding fixed-degree Cayley networks, providing an important connection between permutation generation and network structure.

The case of three distinct prefix reversals is of particular interest, since a generating triple determines a connected cubic spanning subgraph of the pancake graph. Related questions also arise in permutation enumeration. Sawada and Williams \cite{Sawada2016} studied successor rules for generating permutations by prefix reversals and considered the restriction to the three lengths \(n-2\), \(n-1\), and \(n\). Konstantinova and Son \cite{Konstantinova2022} investigated the girths of several families of cubic pancake graphs. These studies motivate a systematic understanding of which triples of prefix reversals generate the full symmetric group.

Since every prefix reversal of length less than \(n\) fixes the position \(n\), any generating triple must contain \(r_n\). The problem therefore reduces to determining when
\[
G_{n}:=\langle r_n,r_m,r_k\rangle=S_n,
\qquad 2\le k<m<n.
\]
Blanco and Golubyatnikov \cite{Blanco2025} obtained complete classifications for triples containing at least one of \(r_2\), \(r_3\), \(r_{n-1}\), and \(r_{n-2}\), together with further necessary conditions for generation.

Their computational investigations also suggested relations between generating triples in different degrees. In particular, they conjectured that generation is preserved when all three reversal lengths are increased by four, and that a corresponding shift by two preserves generation when \(n\equiv0,1\pmod4\). They further proposed a non-generation condition describing a triangular region in the parameter space. These conjectures concern structural features extending beyond the individual families covered by the existing classification results.

The present paper establishes the two shift properties and the parity-dependent non-generation condition. Thus, our results explain how generating triples give rise to further generating triples in larger degrees and provide restrictions on the parameters for which generation is possible. The precise statements are given in the following theorems.

\section{Preliminaries}
Let \(\Omega=\{1,2,\ldots,n\}\), and denote by  \(\operatorname{Alt}_n\) the  alternating groups on \(\Omega\). We use cycle notation for permutations. In cycle notation, \((a_1\,a_2\,\cdots\,a_\ell)\) maps \(a_i\) to \(a_{i+1}\) for \(1\leq i<\ell\), maps \(a_\ell\) to \(a_1\), and fixes all remaining elements of \(\Omega\). Products of permutations are evaluated from right to left; thus
\[
(\sigma\tau)(i)=\sigma(\tau(i))
\qquad
(\sigma,\tau\in\operatorname{Sym}_n,\ i\in\Omega).
\]
Throughout the proofs, we use suitable products of prefix reversals to construct permutations with prescribed cycle structures.

We next recall some terminology concerning permutation groups. Let \(G\) be a finite group acting on \(\Omega\). The orbit of \(x\in\Omega\) under \(G\) is
\[
Gx=\{g(x):g\in G\}.
\]
A subset of \(\Omega\) is \(G\)-invariant if and only if it is a union of \(G\)-orbits. The action is transitive if, for every \(x,y\in\Omega\), there exists \(g\in G\) such that \(g(x)=y\).

A nonempty subset \(B\subseteq\Omega\) is a block for the action of \(G\) if, for every \(g\in G\), either
\[
g(B)=B
\qquad\text{or}\qquad
g(B)\cap B=\varnothing,
\]
where \(g(B)=\{g(b):b\in B\}\). If the action is transitive, the distinct translates of \(B\) form a \(G\)-invariant partition of \(\Omega\), called a block system. The singleton subsets and \(\Omega\) itself are the trivial blocks. A transitive action is primitive if it admits no nontrivial blocks, and imprimitive otherwise. We write \(H<G\) to indicate that \(H\) is a proper subgroup of \(G\).

For ease of reference in the subsequent proofs, we introduce the following notation:
\[
a=n-m,\quad b=n-k,\quad N\in\left\{n+2,n+4\right\},\quad s=a+c, \quad t=N-s.
\]

\begin{lemma}
[{\cite[2.2]{Douglas2003}}]
\label{Douglas}
    Necessary conditions for a set of flips to generate \(S_{n}\).
    
    In order for a set of flips to be able to generate all permutations, all of the following conditions must be satisfied: (these are not sufficient conditions) 
    \begin{enumerate}[label=(\arabic*)]
        \item \label{Douglas.1}
        One of the flips must be of size \(n\).
        \item \label{Douglas.2}
        One of the flips must be of size \(i\), where \(\left \lfloor\frac{n}{2}\right \rfloor<i<n\).Otherwise, there is no way to move an element from the first half to the second half.
        \item \label{Douglas.3}
        One of the flips must be of size \(l\), where \(l\) is even. Otherwise, no element can be moved from an odd-numbered position to an even-numbered position.
        \item \label{Douglas.4}
        One of the flips must be an odd permutation. Otherwise, one can only get the set of all even permutations. 
        \item \label{Douglas.5}
        The set of flips must include three flips with relatively prime sizes. Otherwise, if \(l>1\) is the greatest common divisor of the flips sizes, then some of the blocks of \(l\) contiguous elements can never be separated from each other.
    \end{enumerate}
\end{lemma}

\noindent\textbf{Remark.} Originally, in \cite{Douglas2003} the authors formulate Observation (5) as follows: The generating set must contain at least two prefix reversals with relatively prime indices.  Of course, if this holds, let the set of flips as \(\left\{r_{i},r_{j},r_{k}\right\}\) then \(\gcd(i,j,k)=1\). The converse does not hold. Indeed, if \(i=15,\ j=12\) and \(k=10\) then \(\gcd(i,j)=3\), \(\gcd(j,k)=2\), and \(\gcd(i,k)=5\), however \(\gcd(i,j,k)=1\). Moreover, it is easy to check that \(\left \langle r_{15},r_{12},r_{10} \right \rangle =S_{15}\).

\begin{lemma}[{\cite[Theorem~1.2]{jones2014}}]\label{jones1.2}
   Let \(G\) be a primitive permutation group of finite degree \(n\), not containing the alternating group \(A_{n}\). Suppose that \(G\) contains a cycle fixing \(k\) points, where \(0\le k\le n-2\). Then one of the following holds:
   \begin{enumerate}[label=(\arabic*)]
        \item \label{jones1}
        \(k=0\) and either
        \begin{enumerate}
            \item  \(C_{p}\le G\le AGL_{1}(p)\) with \(n=p\) prime, or
            \item \(PGL_{d}(q)\le G\le P\Gamma L_{d}(q)\) with \(n=(q^d-1)(q-1) \) and \(d\ge 2\) for some prime power \(q\), or
            \item  \(G=L_{2}(11)\), \(M_{12}\) or \( M_{23}\) with \(n=11, 11 \ \text{or} \ 23 \) respectively.
        \end{enumerate}
        \item \label{jones2}
        \(k=1\) and either
        \begin{enumerate}
            \item \(AGL_{d}(q)\le G \le A\Gamma L_{q}\) wwith \(n=q^d\) and \(d\ge 1\) for some prime power q, or
            \item \(G=L_{2}(p)\) or \(PGL_{2}(p)\) with \(n=p+1\) for some prime \(p\ge 5\), or
            \item \(G=M_{11},\ M_{12}\) or \(M_{24}\) with \(n=12,12\ \text{or} \ 24 \) respectively.
        \end{enumerate}
        \item \label{jones3}
        \(k=2\) and \(PGL_{2}(q)\le G\le P\Gamma L_{2}(q)\) with \(n=q+1\) for some prime power q.
    \end{enumerate}
\end{lemma}

\begin{lemma}[{\cite[Corollary~1.3]{jones2014}}]\label{jones1.3}
    Let \(G\) be a primitive permutation group of finite degree \(n\), containing a cycle with k fixed points. Then \(G\ge A_{n}\) if \(k\ge 3\), or if \(k=0,1\ \textit{or}\ 2\) and n avoids the values listed in parts \ref{jones1}, \ref{jones2} or \ref{jones3} of lemma~\ref{jones1.2}. 
\end{lemma}

\begin{lemma}[{\cite[Theorem~1.3]{Iradmusa2018}}]\label{Iradmusa}
    Let \(G\) be the permutation subgroup of \(S_{n}\) generated by the n-cycle \(\sigma:=\left(1\ 2\cdots  n\right)\) and the 3-cycle \(\delta:=(a\ b\ c)\) with even (respectively n odd) and \(1\le a<c<b\le n\). Then, \(G=S_{n}\) (respectively \(G=A_{n}\)) if and only if
    \(\gcd (b-a,c-a,n)=1\).
\end{lemma}

\begin{lemma}[{\cite[Theorem~1.1]{jones2014}}]\label{jones1.1}
    Let \(G\) be a primitive permutation group of finite degree \(n\), containing a cycle of prime length fixing at least three points. Then \(G\ge A_{n}\).
\end{lemma}

\section{Main results}
We now prove the main theorems.
\begin{theorem}\label{m+k<n}
    If $m+k<n-1$, then $\left \langle r_{n},r_{m},r_{k} \right \rangle <S_{n}$. Moreover, if n is even and $m+k<n$, then $\left \langle r_{n},r_{m},r_{k} \right \rangle <S_{n}$.
\end{theorem}
\begin{proof}
    We construct an undirected graph $X$, with vertex $V(X)=\Omega=\left\{1,2,\ldots,n\right\}$  and edge set defined as follows: for $q\in \left\{n,m,k\right\}$ and integers $i$ satisfying $1\le i<q+1-i\le q$, join q and q+1-i. Since $r_{q}=(1\ q)(2\ q-1)\ldots$, the number of edges generated by $r_{q}$ is $\left \lfloor  \frac{q}{2} \right \rfloor $. Therefore, the total number of edges in graph $X$ is $\left | E(X) \right | =\left \lfloor  \frac{n}{2} \right \rfloor+\left \lfloor  \frac{m}{2} \right \rfloor+\left \lfloor  \frac{k}{2} \right \rfloor$.

    In graph $X$, if $r_{q}(x)\ne x$, then $x$ is adjacent to $r_{q}$. Consequently, for any generator $r_{q}$, a point $x$ can either move in one direction or remain fixed. In contrast, if $x$ and  $y$ are adjacent, there exists $q\in \left\{n,m,k\right\}$ such that $y=r_{q}$.Hence, two vertices belong to the same orbit if and only if they belong to the same connected component. Therefore, if $X$ is disconnected, $G_{n}$ does not act transitively on $\Omega$, implying $G_{n}\ne S_{n}$.

    Next,  we consider the parity of $n$.
    
    When $n$ is odd, let $n=2l+1$.
    \begin{gather*}
        \left|E(X)\right|=l+\left \lfloor  \frac{m}{2} \right \rfloor+\left \lfloor  \frac{k}{2} \right \rfloor\le l+\left \lfloor  \frac{m+k}{2} \right \rfloor\le l+\left \lfloor  \frac{2l-1}{2} \right \rfloor=n-2.
    \end{gather*}
    Since a connected graph with $n$ vertices has at least $n-1$ edges, but $X$ has at most $n-2$ edges, $X$ is disconnected. Thus $G_{n}\ne S_{n}$.

    When $n$ is even, let $n=2l$.
    \begin{gather*}
        \left|E(X)\right|=l+\left \lfloor  \frac{m}{2} \right \rfloor+\left \lfloor  \frac{k}{2} \right \rfloor\le l+\left \lfloor  \frac{m+k}{2} \right \rfloor\le l+\left \lfloor  \frac{2l-1}{2} \right \rfloor=n-1.
    \end{gather*}
    If $\left|E(X)\right|<n-1$, $X$ is disconnected, then $G\ne S_{n}$. In the following, we suppose that $X$ is connected. Then $\left|E(X)\right|=n-1$, implying that $X$ is a tree. Let $v_{0}$ be the root of the tree, and denote the unique path from $v_{0}$ to v by $p(v)$.

    Define $c_{v}$ to be the parity (modulo 2) of the number occurrences of $r_{n}$ alone the path $p_{v}$.Then let
    \begin{gather*}
        B_{0}=\left\{v\in \Omega:c_{v}=0\right\},\qquad  B_{1}=\left\{v\in \Omega:c_{v}=1\right\}.
    \end{gather*}
    For $m$ and $k$, the parity of the number of occurrences of $r_{n}$ remains unchanged. Thus
    \begin{gather*}
        c\left(r_{m}(v)\right)=c(v),\qquad c\left(r_{k}(v)\right)=c(v).
    \end{gather*}
    Therefore,
    \begin{gather*}
        r_{m}(B_{0})=B_{0},\qquad r_{m}(B_{1})=B_{1},\qquad r_{k}(B_{0})=B_{0},\qquad r_{k}(B_{1})=B_{1}.
    \end{gather*}
    Furthermore, since n is even, $r_{n}$ has no fixed points. Consequently, every move by $r_{n}$ in the graph $X$ changes the number of $c_{v}$ exactly once. Therefore,
    \begin{gather*}
        c(r_{n}(v)) \equiv c(v)+1 \pmod{2} .
    \end{gather*}
    Hence
    \begin{gather*}
        r_{n}(B_{0})=B_{1},\qquad r_{n}(B_{1})=B_{0}.
    \end{gather*}
    In conclusion, $B_{1}$ and $B_{0}$ are two imprimitive blocks of $\Omega$. Therefore $G_{n}\ne S_{n}$. 
    
    Consequently, Theorem holds in both the even and odd cases.
\end{proof}

\noindent\textbf{Remark.} Originally, in \cite{Blanco2025} the author formulate Conjecture 4 as follow: If\(m+k<n\) then $\left \langle r_{n},r_{m},r_{k} \right \rangle <S_{n}$. Moreover, if \(n\) is odd and \(m+k<n-1\) then $\left \langle r_{n},r_{m},r_{k} \right \rangle <S_{n}$. For \(n\) odd, this formulation may be ambiguous. Thus we reformulate the statement.

To prove the remaining two problems, we first establish the following lemma.

\begin{lemma}
      If \(G_{n} =S_{n}\), then \(\gcd(a,b)=1\).
\end{lemma}
\begin{proof}
    We first consider action
    \[
    T_a=r_{n-a}r_n=r_mr_n,\qquad T_b=r_{n-b}r_n=r_kr_n.
    \]
    Since \(r_n^2=1\), we have
    \[
    r_m=T_ar_n,\qquad r_k=T_br_n,
    \]
    and hence
    \[
    G_n=\langle r_n,r_m,r_k\rangle
    =\langle r_n,T_a,T_b\rangle.
    \]
    Let \(K=\langle T_a,T_b\rangle\). Observe that
    \[
    r_nT_ar_n=r_nr_m=T_a^{-1}\in K,
    \qquad
    r_nT_br_n=r_nr_k=T_b^{-1}\in K.
    \]
    Thus \(r_n\) normalizes \(K\), so \(K\trianglelefteq G_n\). Since \(G_n=S_n\), it follows that
    \[
    A_n\le K.
    \]
    In particular, \(K\) is transitive on \(\{1,\ldots,n\}\).

    Suppose, for a contradiction, that \(g=\gcd(a,b)>1\). For each \(d\in\{a,b\}\), we have \(g\mid d\), and the action of \(T_d\) is given by
    \[
    T_d(x)=
    \begin{cases}
        x+d, & 1\le x\le n-d,\\
        n+1-x, & n-d<x\le n.
    \end{cases}
    \]
    Consequently,
    \[
    T_d(x)\equiv
    \begin{cases}
        x \pmod g, & 1\le x\le n-d,\\
        n+1-x \pmod g, & n-d<x\le n.
    \end{cases}
    \]
    Consider the subset
    \[
    B_0=\bigl\{x\in\{1,\ldots,n\}:
    x\equiv1\pmod g
    \text{ or }
    x\equiv n\pmod g\bigr\}.
    \]
    The map \(x\mapsto n+1-x\) interchanges the residue classes of \(1\) and \(n\) modulo \(g\). It follows that
    \[
    T_a(B_0)=B_0,
    \qquad
    T_b(B_0)=B_0.
    \]
    Thus \(B_0\) is \(K\)-invariant. If \(g\ge3\), then \(B_0\) is a nonempty proper subset of \(\{1,\ldots,n\}\), since it contains at most two residue classes modulo \(g\). This contradicts the transitivity of \(K\).

    It remains to consider \(g=2\). In this case, both \(a\) and \(b\) are even, so \(n,m,k\) have the same parity. If \(n\) is odd, then all three indices are odd, contrary to lemma~\ref{Douglas}~\ref{Douglas.3}, which requires at least one of them to be even. If \(n\) is even, then all three indices are even, contrary to lemma~\ref{Douglas}~\ref{Douglas.5}, which requires \(\gcd(n,m,k)=1\). Hence \(g=2\) is also impossible.

    We conclude that \(\gcd(a,b)=1\).
\end{proof}

\begin{lemma}\label{An}
    If \(G_{n}=S_{n}\) and \(t\ge 3\), assume additionally that \(N\) is odd whenever both \(a\) and \(b\) are odd, then \(A_{n}\le G_{N}\).
\end{lemma}
\begin{proof}
    We first consider action
    \[
    D=T_{a}T_{b}T_{a}^{-1}T_{b}^{-1} .
    \]
    A direct computation yields
    \[
    D(x)=
    \begin{cases}
        x+2a, & 1\le x\le b-a,\\
        x+a-b, & b-a<x\le a+b,\\
        x, & a+b<x\le n.
    \end{cases}
    \]
    hence
    \[
    D(x)\equiv
    \begin{cases}
        x+2a \pmod s, & 1\le x\le s,\\
        x \pmod s, & s<x\le n.
    \end{cases}
    \]
    Taking a suitable power of \(D\), we obtain
    \begin{equation}\label{U}
    U=
    \begin{cases}
        (1\,2\,\cdots\,s), & \text{if \(a\) and \(b\) have opposite parity},\\[2mm]
        (1\,3\,\cdots\,s-1)(2\,4\,\cdots\,s), & \text{if \(a\) and \(b\) are both odd}.
    \end{cases}
    \end{equation}
    Let
    \[
    S=\left\{1,2\ldots ,s \right\},\qquad F=\left\{s+1,s+2,\ldots ,N\right\}.
    \]

    We now distinguish two cases according to the parities of \(a\) and \(b\).

    \begin{enumerate}[label=(\arabic*)]
    \item \(a\) and \(b\) have opposite parity.
    
    In this case,
    \[
    U=(1\,2\,\cdots\,s).
    \]
    We prove directly that \(G_{N}\) is primitive.

    Suppose that a proper block \(B\) meets both \(S\) and \(F\). Since \(U\) fixes every point of \(F\), any point of \(B\cap F\) belongs to both \(U(B)\) and \(B\). Thus
    \[
    U(B)\cap B\neq\varnothing,
    \]
    and the defining property of a block implies that \(U(B)=B\). As \(B\cap S\neq\varnothing\) and \( U\) acts transitively on \(S\), it follows that
    \[
    S\subseteq B.
    \]
    Hence,
    \[
    T_a(S)\cap S\neq\varnothing,
    \qquad
    T_b(S)\cap S\neq\varnothing.
    \]
    Consequently, \(T_a(B)\cap B\neq\varnothing\) and \(T_b(B)\cap B\neq\varnothing\), whence
    \[
    T_a(B)=B,
    \qquad
    T_b(B)=B.
    \]
    Thus \(B\) is invariant under \(K=\langle T_a,T_b\rangle\). The transitivity of \(K\) then forces \(B=\Omega\), contradicting the assumption that \(B\) is proper.

    Now suppose that \(G_{N}\) has a nontrivial block system \(\mathcal B\). By the preceding argument, every block in \(\mathcal B\) is contained entirely in either \(S\) or \(F\). Hence \(S\) is a union of blocks of \(\mathcal B\). Set
    \[
    h=\min(a,t),
    \qquad
    I=S\setminus T_a(S)=\{1,\ldots,h\}.
    \]
    Since \(H\) permutes the blocks of \(\mathcal B\), the set \(T_a(S)\) is also a union of blocks. Taking the difference shows that \(I\) is a union of blocks, as is \(U(I)\). Moreover,
    \[
    U(I)=\{2,\ldots,h+1\},
    \qquad
    I\setminus U(I)=\{1\}.
    \]
    It follows that \(\{1\}\) is a union of blocks and therefore is itself a block of \(\mathcal B\). By transitivity, all blocks in \(\mathcal B\) are singletons, contradicting its nontriviality. Thus \(G_{N}\) is primitive.

    Finally, \(G_{N}\) contains the \(s\)-cycle \(U\), which fixes \(t\geq3\) points. By 
    lemma~\ref{jones1.3},
    \[
    G_{N}\geq A_N.
    \]
    \item  \(a\) and \(b\) are both odd. Since \(\gcd(a,b)=1\), we have
    \[
    \gcd(a-b,a+b)=2.
    \]
    Taking a suitable power of \(D\), we therefore obtain
    \[
    U=(1\ 3\ 5\ \cdots\ s-1)(2\ 4\ 6\ \cdots\ s).
    \]
    Set
    \[
    P=\{1,3,\ldots,s-1\},\qquad
    Q=\{2,4,\ldots,s\},\qquad
    \Omega'=\{1,2,\ldots,N\}.
    \]
    Thus \(U\) acts as a cycle on each of \(P\) and \(Q\). We first show that any block containing either \(P\) or \(Q\) must equal \(\Omega'\).

    Suppose \(B\) be a block with \(P\subseteq B\), and put
    \(B'=T_a(B)\).
    Since  
    \[
    T_a(1)=1+a\in Q,
    \qquad
    T_a(s+2-a)=s+2\in F,
    \]
    the block \(B'\) meets both \(Q\) and \(F\). As \(U\) fixes every point of \(F\), we have \(U(B')\cap B'\ne\varnothing\), and hence \(U(B')=B'\). Since \(U\) is transitive on \(Q\), it follows that \(Q\subseteq B'\).

    Moreover, \(a<b\), and both \(a\) and \(b\) are odd, so \(b\ge a+2\). Consequently,
    \[
    T_a^2(1)=1+2a\in P\subseteq B.
    \]
    Thus \(T_a^2(B)\cap B\ne\varnothing\), which implies
    \[
    T_a(B')=T_a^2(B)=B.
    \]
    Similarly,
    \[
    T_b(1)=b+1\in Q\subseteq B',
    \qquad
    T_b(2)=b+2\in P\subseteq B,
    \]
    and therefore
    \[
    T_b(B)=B',
    \qquad
    T_b(B')=B.
    \]
    It follows that \(B\cup B'\) is invariant under both \(T_a\) and \(T_b\), and hence under \(K\). By the transitivity of \(K\),
    \[
    B\cup B'=\Omega'.
    \]
    If \(B\ne B'\), then \(B\) and \(B'\) are disjoint blocks of equal size, forcing \(N\) to be even. This contradicts the assumption that \(N\) is odd. Hence
    \[
    B=B'=\Omega'.
    \]
    The case \(Q\subseteq B\) follows by a similar argument.

    Now let \(B\) be a block meeting both \(S\) and \(F\). Since \(U\) fixes every point of \(F\), the block \(B\) is \(U\)-invariant. As \(B\) meets \(S=P\cup Q\), it must contain either \(P\) or \(Q\). The preceding argument then gives \(B=\Omega'\).

    Set
    \[
    h=\min\{a,t\},
    \qquad
    I=S\cap T_a(S)=\{1,2,\ldots,h\}.
    \]
    Suppose first that \(h\ge3\). Then \(h\) is odd and \(h+2\le s\), so
    \[
    U(I)=\{3,4,\ldots,h+2\},
    \qquad
    I\setminus U(I)=\{1,2\}.
    \]
    Thus \(\{1,2\}\) is a union of blocks. If \(c\) denotes the common block size, then \(c\mid2\) and \(c\mid N\). Since \(N\) is odd, we obtain \(c=1\). Therefore \(G_N\) is primitive.

    We now complete the argument for \(a\ge3\). Write
    \[
    t=N-s\ge3,
    \qquad h\le a<b.
    \]
    We construct an element \(V\in G_N\) supported on the interval
    \[
    \{h+1,h+2,\ldots,h+s\}
    \]
    whose two cycles consist of alternate points of this interval. If \(t\ge a\), set
    \[
    V=T_aUT_a^{-1}.
    \]
    If \(3\le t<a\), set
    \[
    V=r_NU^{-1}r_N.
    \]
    In either case,
    \[
    V=(h+1\ h+3\ \cdots\ h+s-1)
    (h+2\ h+4\ \cdots\ h+s).
    \]
    Let \(W=VU^{-1}\). A direct calculation yields
    \[
    W=
    (1\ s+1\ s+3\ \cdots\ s+h\ 2\ h-2\ h-4\ \cdots\ 3)
    (2\ s+2\ s+4\ \cdots\ h+s-1\ h+1\ h-1\ \cdots\ 4).
    \]
    The two disjoint cycles in this decomposition have lengths \(h+2\) and \(h\), respectively. Since \(h\) is odd,
    \[
    \gcd(h+2,h)=1.
    \]
    Hence \(W^h\) is a single \((h+2)\)-cycle. The number of points fixed by \(W^h\) is
    \[
    N-(h+2)=s+a-2\ge s-2\ge6.
    \]
    Since \(G_N\) is primitive, lemma ~\ref{jones1.3} now implies
    \[
    A_N\le G_N.
    \]

    It remains to consider \(a=1\) and \(b\) odd. In this case,
    \[
    T_1=(1\ 2\ \cdots\ N),
    \qquad
    U=(1\ 3\ \cdots\ s-1)(2\ 4\ \cdots\ s),
    \]
    and a direct calculation gives
    \[
    T_1UT_1^{-1}U^{-1}=(1\ s+1\ 3).
    \]
    Since \(N\) is odd, we have \(\gcd(s,2,N)=1\). Lemma~\ref{Iradmusa} therefore yields
    \[
    A_N\le G_N.
    \]
    Combining the preceding cases, we obtain \(A_{N}\le G_{n}\) for all admissible values of \(a\) and \(b\). 
    \end{enumerate}
\end{proof}

\begin{theorem}
    If \(G_{n}=S_{n}\), then \(G_{n+4}=S_{n+4}\)
\end{theorem}

\begin{proof}
    By theorem~\ref{m+k<n}, we have \(m+k-1\ge n\), hence \(n+4\ge n-m+n-k+3\), it mains \(t\ge 3\). By lemma~\ref{An}, \(G_{N}\ge A_{N}\). Since \(G_{n}=S_{n}\), at least one of \(r_{n},r_{m},r_{k}\) is an odd permutation, which mean there is an index \(i\in \left\{n,m,k\right\}\) satisfying 
    \(i \equiv 2\ \text{or}\ 3 \pmod4\)    Then \(r_{i+4} \equiv r_{i} \pmod{4}\), there is an odd permutation in \(G_{N}\).By lemma~\ref{An}, \(G_{N}=S_{N}\).
\end{proof}

\begin{theorem}
    For any \(n \equiv 0\ \text{or}\ 1 \pmod 4 \), if \(G_{n}=S_{n}\), then \(G_{n+2}=S_{n+2}\)
\end{theorem}

\begin{proof}
    By theorem~\ref{m+k<n}, we have \(m+k-1\ge n\), hence \(n+2\ge n-m+n-k+1\), it mains \(t\ge 1\). By lemma~\ref{An}, when \(t\ge 3\), \(G_{N}\ge A_{N}\). Hence we consider \(t=1\) and \(t=2\) at the following.
    \begin{enumerate}[label=(\arabic*)]
        \item \(t=2\)
        
    In this case, the hypotheses of theorem give
    \[
    N\equiv3\pmod4,\qquad s=N-2,\qquad N\ge7.
    \]
    In particular, \(s\) is odd, and \(a\) and \(b\) have opposite parity. By ~\ref{U}, the group \(G_{N}\) contains the cycle
    \[
    U=(1\ 2\ \cdots\ s),
    \]
    whose support is \(S=\{1,\ldots,s\}\) and whose fixed-point set is
    \(F\).

    We first prove that \(G_{N}\) is primitive. Let \(\mathcal B\) be an \(G_{N}\)-invariant partition of \(\Omega\) with at least two parts. Suppose that some block \(B\in\mathcal B\) meets both \(S\) and \(F\). Since \(U\) fixes every point of \(B\cap F\), we have
    \[
    U(B)\cap B\ne\varnothing.
    \]
    The block property therefore implies \(U(B)=B\). As \(B\cap S\ne\varnothing\) and \(U\) acts transitively on \(S\), it follows that \(S\subseteq B\).

    For each \(d\in\{a,b\}\),
    \[
    T_d(S)\cap S\ne\varnothing.
    \]
    Since \(S\subseteq B\), this implies \(T_d(B)\cap B\ne\varnothing\), and hence \(T_d(B)=B\). Thus \(B\) is invariant under \(K=\langle T_a,T_c\rangle\). The transitivity of \(K\) now gives \(B=\Omega\), contradicting the choice of \(\mathcal B\).

    Consequently, every block of \(\mathcal B\) is contained entirely in either \(S\) or \(F\). In particular, \(S\) is a union of blocks. Let
    \[
    h=\min\{a,2\}.
    \]
    Using the explicit action of \(T_a\), we obtain
    \[
    I:=S\setminus T_a(S)=\{1,\ldots,h\}.
    \]
    Because \(\mathcal B\) is \(G_{N}\)-invariant and \(T_a\in G_{N}\), both \(S\) and \(T_a(S)\) are unions of blocks of \(\mathcal B\). Their difference \(I\) is therefore also a union of blocks. Similarly, \(U(I)\) is a union of blocks. Since \(h\in\{1,2\}\) and \(U=(1\ 2\ \cdots\ s)\), we have
    \[
    I\setminus U(I)=\{1\}.
    \]
    Thus \(\{1\}\) is a union of blocks and must itself be a block. Since \(G_{N}\) is transitive, all blocks in \(\mathcal B\) have the same size. Hence every block is a singleton. This proves that \(G_{N}\) is primitive.

    Suppose, for a contradiction, that \(A_N\nleq G_{N}\). Since \(G_{N}\) is primitive and contains the \((N-2)\)-cycle \(U\), the theorem~\ref{jones1.2}~\ref{jones3} implies that
    \[
    N=q+1,
    \]
    for some prime power \(q\). However,
    \[
    q=N-1\equiv2\pmod4.
    \]
    Thus \(q\) is an even prime power, so \(q=2^e\) for some integer \(e\ge1\). The congruence \(2^e\equiv2\pmod4\) forces \(e=1\), giving \(N=3\). This contradicts \(N\ge7\). Therefore
    \[
    A_N\le G_{N}.
    \]

        \item \(t=1\)
     In this case,
    \[
    s=a+b=N-1=n+1.
    \]
    Under the hypotheses under consideration, \(n\) is odd, so \(s\) is even. Since \(\gcd(a,b)=1\), it follows that \(a\) and \(b\) are both odd. Moreover,
    \[
    k=n-b=n-(n+1-a)=a-1.
    \]
    The assumption \(k\ge2\) therefore gives \(a\ge3\).

    By \ref{U}
    \[
    U=(1\ 3\ \cdots\ s-1)(2\ 4\ \cdots\ s)
    \]
    acts transitively on each of \(P\) and \(Q\) and fixes \(N\).

    We first claim that any block for \(G_{N}\) containing either \(P\) or \(Q\) must equal \(\Omega\). Suppose first that \(Q\subseteq B\), and put \(B'=T_a(B)\). Since \(N-a=b+1\), the explicit formula for \(T_a\) gives
    \[
    T_a(c+1)=N,\qquad
    T_a(2)=a+2\in P,\qquad
    T_a(s)=2\in Q.
    \]
    All three points \(b+1,2,s\) belong to \(Q\). Thus \(B'\) contains \(N\) and meets both \(P\) and \(Q\). Since \(U(N)=N\), we have
    \[
    U(B')\cap B'\ne\varnothing,
    \]
    and the block property yields \(U(B')=B'\). The transitivity of \(U\) on each of \(P\) and \(Q\) now implies
    \[
    P\cup Q\cup\{N\}\subseteq B'.
    \]
    Hence \(B'=\Omega\), and therefore \(B=\Omega\).

    Next, suppose that \(P\subseteq B\). Since \(a\ge3\),
    \[
    s-1=a+b-1>b+1=N-a,
    \]
    so
    \[
    T_a(s-1)=N+1-(s-1)=3\in P.
    \]
    Consequently, \(T_a(B)\cap B\ne\varnothing\), and hence \(T_a(B)=B\). In particular,
    \[
    T_a(1)=a+1\in Q\cap B.
    \]
    Furthermore, \(U(P)=P\subseteq B\), so \(U(B)\cap B\ne\varnothing\), giving \(U(B)=B\). Since \(B\) meets \(Q\) and \(U\) acts transitively on \(Q\), we obtain \(Q\subseteq B\). The preceding argument then gives \(B=\Omega\), proving the claim.

    We now establish primitivity. Let \(\mathcal B\) be an \(G_{N}\)-invariant partition of \(\Omega\) with at least two parts, and let \(B\in\mathcal B\) contain \(N\). Since \(U\) fixes \(N\), the block \(B\) is \(U\)-invariant. If \(B\) contained a point of \(S\), it would contain that point's entire orbit under \(U\), and hence either \(P\) or \(Q\). By the claim, this would force \(B=\Omega\), a contradiction. Therefore \(B=\{N\}\). As \(G_{N}\) is transitive, all blocks of \(\mathcal B\) have the same size, so every block is a singleton. Thus \(G_{N}\) is primitive.

    It remains to construct a \(3\)-cycle in \(G_{N}\). Set
    \[
    V=r_NU^{-1}r_N.
    \]
    Conjugation by \(r_N\) gives
    \[
    V=(2\ 4\ \cdots\ s)(3\ 5\ \cdots\ N).
    \]
    Let
    \[
    A=(1\ 3\ \cdots\ s-1),\qquad
    B_0=(2\ 4\ \cdots\ s),\qquad
    C=(3\ 5\ \cdots\ N).
    \]
    Then \(U=AB_0\) and \(V=B_0C\). Since \(B_0\) has support disjoint from those of both \(A\) and \(C\), it commutes with both, and therefore
    \[
    VU^{-1}
    =B_0CB_0^{-1}A^{-1}
    =CA^{-1}.
    \]
    With products acting from left to right, we have
    \[
    1\xrightarrow{A^{-1}}s-1\xrightarrow{C}N,\qquad
    N\xrightarrow{A^{-1}}N\xrightarrow{C}3,\qquad
    3\xrightarrow{A^{-1}}1\xrightarrow{C}1.
    \]
    Every even point is fixed by both \(A^{-1}\) and \(C\). For each remaining odd point \(x\in\{5,7,\ldots,s-1\}\),
    \[
    x\xrightarrow{A^{-1}}x-2\xrightarrow{C}x.
    \]
    Consequently,
    \[
    U^{-1}V=(1\ N\ 3)\in H.
    \]

    Finally, since \(a\ge3\) and \(b>a\) are both odd, we have \(b\ge5\), and hence
    \[
    N=a+c+1\ge9.
    \]
    Thus \(G_{N}\) is primitive and contains a \(3\)-cycle with at least six fixed points. Theorem~\ref{jones1.1} therefore implies
    \[
    A_N\le G_{N}.
    \]

    \end{enumerate}
    Combining the preceding cases, we obtain \(A_{N}\le G_{N}\) for \(t\ge 1\).Since \(n \equiv 0\ \text{or}\ 1 \pmod 4 \), \(N \equiv 2\ \text{or}\ 3 \pmod 4 \). Thus \(r_{N}\) is an odd permutation in \(G_{N}\), and then \(G_{N}=S_{N}\).      
\end{proof}

\section*{Acknowledgments}
    Supported by the National Natural Science Foundation of China (No.12661002) ,Inner Mongolia Natural Science Foundation (No. 2025MS01022) and Hohhot Science and Technology Program(No. 2025-Gui-Ji-59).

\section*{Data availability}
    No data was used for the research described in the article.

\end{document}